\documentclass[a4paper, 12pt]{amsart}

\usepackage{amssymb}
\usepackage{amsthm}
\usepackage{amsmath}

\usepackage{comment}

\usepackage{url}

\usepackage[mathscr]{eucal}

\usepackage{enumitem}

\theoremstyle{plain}
\newtheorem{thm}{Theorem}[section]

\newtheorem{lem}[thm]{Lemma}

\theoremstyle{definition}

\newtheorem{ques}[thm]{Question}

\theoremstyle{remark}
\newtheorem{rmk}{Remark}[section]
\newtheorem*{ac}{Acknowledgements}

\newcommand{\zz}{\mathbb{Z}}

\newcommand{\rr}{\mathbb{R}}

\newcommand{\yoemph}[1]{\emph{#1}}

\newcommand{\supp}{\text{{\rm supp}}}

\newcommand{\yosub}{\subseteq}

\newcommand{\yocmet}[2]{\mathrm{UCPM}(#1, #2)}

\newcommand{\yocpm}[1]{\mathrm{CPM}(#1)}

\newcommand{\met}[1]{\mathrm{Met}(#1)}
\DeclareMathOperator{\metdis}{\mathcal{D}}

\newcommand{\ult}[2]{\mathrm{UMet}(#1; #2)}
\DeclareMathOperator{\umetdis}{\mathcal{UD}}

\newcommand{\compm}[1]{\mathrm{Comp}(#1)}

\newcommand{\ucompm}[2]{\mathrm{UComp}(#1; #2)}

\newcommand{\yocover}[1]{\mathcal{#1}}

\newcommand{\yoellsp}[1]{\ell^{1}(#1)}
\newcommand{\yoelldis}[1]{\|#1\|_{\ell^{1}}}

\newcommand{\yoles}[1]{\mathrm{L}(#1)}

\newcommand{\youles}[2]{\mathrm{UL}(#1; #2)}

\newcommand{\yobasis}{\mathbf{e}}

\newcommand{\yoiset}[1]{\mathrm{I}(#1)}
\newcommand{\youiset}[2]{\mathrm{UI}(#1, #2)}

\newcommand{\yocase}[2]{Case #1.~[#2]:}

\begin{document}


\title[Spaces of metrics are Baire]
{
Spaces of metrics are Baire
}

\author[Yoshito Ishiki]
{Yoshito Ishiki}

\address[Yoshito Ishiki]
{\endgraf
Department of Mathematical Sciences
\endgraf
Tokyo Metropolitan University
\endgraf
Minami-osawa Hachioji Tokyo 192-0397
\endgraf
Japan
}

\email{ishiki-yoshito@tmu.ac.jp}


\date{\today}



\begin{abstract}
For a metrizable space, 
we consider the space of all metrics 
generating the same topology of the metrizable space, 
and 
 this space of metrics is equipped with 
the supremum metric. 
In this paper, 
for every metrizable space, 
we establish 
 that 
the space of metrics on the metrizable space is  Baire. 
We also 
show that 
the set of all complete metrics 
is  comeager in the space of metrics. 
Moreover, we investigate non--Archimedean analogues of 
these results. 
\end{abstract}

\maketitle

\section{Introduction}\label{sec:intro}
\subsection{Backgrounds}

A subset 
$S$ 
of a topological space
 $M$
 is said to be 
 \yoemph{comeager}
 or 
 \yoemph{residual}
 if 
 there exists 
 a sequence 
 $\{G_{n}\}_{n\in \zz_{\ge 0}}$
 of  open 
 dense 
 subsets of 
 $M$
  such that 
  $\bigcap_{n\in \zz_{\ge 0}}G_{n}\yosub S$. 
A non-empty topological space 
 $X$ 
 is 
 \yoemph{Baire} if 
 every comeager subset of 
 $X$
  is 
 dense in 
 $X$. 
The concept of Baire spaces 
goes back to Baire's paper 
\cite{baire1899fonctions}, 
and 
provides
a 
powerful 
framework 
 to 
verify 
the 
 denseness of subsets of a space.
For instance, 
Banach and Mazurkiewicz's  method 
\cite{Banach1931, Mazurkiewicz1931}, 
which demonstrates  the existence
and denseness  of 
nowhere differentiable continuous functions 
using Baire spaces, 
 is 
 a
notable 
and 
succinct 
illustration
  of the 
  efficacy 
  of Baire spaces. 

In this paper, 
  for a   metrizable topological space
  $X$, 
we consider 
the space 
$\yocpm{X}$ 
of 
all continuous pseudometrics on
 $X$, and 
the space
$\met{X}$ 
of all metrics 
generating the same topology of
 $X$. 
These  spaces are equipped with 
the supremum metric
$\metdis_{X}$. 
For a metrizable space 
$X$, 
and for an open cover 
$\yocover{C}$ 
of 
$X$, 
we first  show the openness and denseness of 
the set of all 
$d\in \yocpm{X}$
 such that there exists a Lebesgue number 
 of
 $\yocover{C}$
  with respect to 
  $d$
  (Theorem \ref{thm:main00}). 
 As a consequence, 
 for a fixed metric $w\in \yocpm{X}$, 
 we prove the denseness and $G_{\delta}$-ness of  the set of all $d\in \yocpm{X}$
 such that  $1_{X}\colon (X, d)\to (X, w)$ is uniformly continuous
 (Theorem \ref{thm:mainiset}). 
 As applications of this theorem, 
for every metrizable space
$X$, 
we establish 
 that 
the space
$(\met{X}, \metdis_{X})$ of metrics on the metrizable space is  Baire
(Theorem \ref{thm:main1baire}). 
We also 
show that 
the set of all complete metrics 
is  comeager in  $\met{X}$
(Theorem \ref{thm:maincomp}). 
Moreover, we investigate non--Archimedean analogues of 
these results. 

Since  
$(\met{X}, \metdis_{X})$ 
is a moduli space of 
metrics on 
$X$, 
our main results provide a method  for demonstrating the 
existence and abundance  of special metrics on 
$X$
through the theory of Baire spaces. 
Specifically, 
we can ascertain the variety of 
``geometries'' 
that can be developed on  on $X$. 
From the  point of view of 
analogues  between measures and Baire's
 categories as mentioned in Oxtoby's book \cite{MR0584443}, 
by the help of our main result 
(Theorem \ref{thm:main1baire})
asserting that $\met{X}$ is Baire, 
 the author's works
 (\cite{Ishiki2020int}, 
\cite{Ishiki2021ultra}, 
\cite{Ishiki2021dense}, 
\cite{MR4527953}, 
\cite{Ishiki2023disco},
and 
\cite{Ishiki2023sr})
 on spaces of metrics, 
 including 
the present paper,   
can 
be considered as  a counterpart of 
Vershik's works
(\cite{MR1691182}
and 
\cite{MR2086637}) on 
measures on the set of metrics.

Next
we briefly review  research on 
spaces of metrics. 
In 1944, 
Shanks
\cite{MR0010962} 
considered  
spaces of metrics on compact metrizable spaces, 
 and 
established 
a
  Banach--Stone--Eilenberg  type theorem 
stating  that for every pair $X$ and $Y$ of compact metrizable spaces, 
 $\met{X}$ is  congruent to $\met{Y}$ if and only if 
 $X$ is 
  homeomorphic to  $Y$(\cite[Theorem 3.2]{MR0010962}). 
  
In the 1990s, 
some 
authors
investigated 
spaces of  all possible  metrics on 
given 
sets
(\cite{MR1205515}, 
\cite{MR1268864}, 
\cite{MR1198097}, 
\cite{vceretkova1997certain}, 
and 
\cite{MR1407287}). 
Remark that 
this 
 space of metrics   depends only on 
the cardinality  of an  underlying  set. 

Starting in  2020, 
in contrast, 
the author of the present paper considered 
the set of topological metrics; namely, 
 for a metrizable space 
 $X$, 
the space 
$\met{X}$ 
of metrics 
 generating the same topology of 
$X$ equipped with the supremum distance
$\metdis_{X}$. 
Although 
it was not known whether 
$\met{X}$ is Baire or not, 
the author 
clarified 
 the denseness and
  Borel hierarchy of 
 a subset 
 $\{\, d\in \met{X}\mid \text{$(X, d)$ satisfies $\mathcal{P}$}\, \}$
 for 
a certain  property
$\mathcal{P}$
  on metric spaces, 
and proved that 
some  subsets are comeager in 
$\met{X}$
(\cite{Ishiki2020int}, 
\cite{Ishiki2021ultra}, 
\cite{Ishiki2021dense}, 
\cite{MR4527953}, 
\cite{Ishiki2023disco},
and 
\cite{Ishiki2023sr}). 
For example, 
the author \cite{Ishiki2020int}
showed that 
the set of all 
metrics in $\met{X}$ having Assouad dimension $\infty$ is 
dense and $G_{\delta}$, in particular,  it is Baire. 

In the context of 
Lipschitz-free metric spaces, 
there are several works on 
spaces of metrics
(see 
\cite{smith2023lipschitz}, 
\cite{MR4641123}, 
and 
\cite[Problem 6.6]{MR3518958}).

As applications of infinite-dimensional topology, 
Koshino 
researched  
topological shapes of spaces of  metrics
equipped with 
not only the uniform topologies 
 but also the compact-open topologies
(\cite{MR4586584}, 
\cite{koshino2022topological}, 
and 
\cite{koshino2024borel}).


 \subsection{Main results}
Before stating our main results, 
we introduce some notions and notations. 
For a set 
$X$, 
 a map 
 $d\colon X\times X\to [0, \infty)$
 is called a 
 \yoemph{pseudometric} if 
 the following conditions are true:
\begin{enumerate}
\item 
for every 
$x\in X$, 
we have 
$d(x, x)=0$; 
\item 
for every pair  
$x, y\in X$, 
we have 
$d(x, y)=d(y, x)$; 
\item 
for every triple 
$x, y, z\in X$, 
we have 
$d(x, y)\le d(x, z)+d(z, y)$. 
\end{enumerate}
 A pair 
 $(X, d)$ is called a 
 \emph{pseudometric space}. 
If the equality
 $d(x, y)=0$ 
 implies 
 $x=y$, 
then 
 $d$ 
 is called  a 
 \yoemph{metric}.

For a 
topological space 
$X$, 
we denote by 
$\yocpm{X}$
 the set of all 
 continuous 
 maps $d\colon X\times X\to [0, \infty)$
such that 
$d$ 
is a pseudometric on  
$X$. 
As mentioned in the previous subsection, 
we also denote by 
$\met{X}$
the set of all 
metrics $d$ on $X$
generating 
the same topology of
$X$. 
Notice that 
$\met{X}\yosub \yocpm{X}$. 
Of course, 
$X$ is metrizable if and only if 
$\met{X}\neq \emptyset$. 
We define 
$\metdis_{X}\colon \yocpm{X}^{2}\to [0, \infty]$
by 
$\metdis_{X}(d, e)=\sup_{x, y\in X}|d(x, y)-e(x, y)|$. 
Note that 
although 
$\metdis_{X}$ 
can take the value 
$\infty$, 
we can define the topology induced by 
$\metdis_{X}$  
using open balls, 
as in the cases of ordinary metrics. 
In this paper, 
we represent
the restricted metric  
$\metdis_{X}|_{\met{X}^{2}}$ as 
the original  symbol 
$\metdis_{X}$. 
In what follows, 
we consider that 
$\yocpm{X}$
 and 
 $\met{X}$
  are equipped with 
  the topologies induced by 
  $\metdis_{X}$, which represent the uniform convergence of metrics.

 For a pseudometric space 
 $(X, d)$, for a point $x\in X$, 
 and for $r\in (0, \infty)$, 
 we denote by 
 $U(x, r; d)$ 
 the open ball 
 centered at $x$
 with radius $x$ of $(X, d)$, 
 i.e., 
 $U(x, r; d)
 =\{\, p\in X\mid d(x, p)<r\, \}$. 
 
 For a pseudometric space 
$(X, d)$, 
and a 
covering
$\yocover{C}=\{C_{i}\}_{i\in I}$
 of 
$X$, 
we say that 
a positive real number 
$r\in (0, \infty)$ is 
a
\yoemph{Lebesgue number
of 
$\yocover{C}$}
if 
for every 
$x\in X$
there exists 
$i\in I$
 such that 
 $U(x, r; d)\yosub C_{i}$.

For a topological space 
$X$, 
and for a 
covering 
$\yocover{C}$
 of
  $X$, 
we denote by 
$\yoles{\yocover{C}}$
the set of all 
$d\in \yocpm{X}$
such that 
$\yocover{C}$
has a 
(positive)
Lebesgue number 
with respect $d$. 

Our first result 
states that
$\yoles{\yocover{C}}$ 
is open and dense in 
the space of 
pseudometrics. 

\begin{thm}\label{thm:main00}
Let 
$X$ be a 
paracompact Hausdorff space, 
and 
$\yocover{C}$
an open 
covering of 
$X$. 
Then 
$\yoles{\yocover{C}}$
 is 
open
and 
dense in 
$\yocpm{X}$. 
\end{thm}

Let $X$ be a metrizable space, 
and $w\in \yocpm{X}$. 
We define 
$\yoiset{w}$
the set of all 
$d\in \yocpm{X}$
such that 
$1_{X}\colon (X, d)\to (X, w)$ is 
uniformly continuous, 
where 
$1_{X}$
 stands for the identity map. 
Namely, 
$d\in \yoiset{w}$
if and only if for every 
$\epsilon\in (0, \infty)$, 
there exists 
$\delta\in (0, \infty)$ such that 
for every pair 
$x, y\in X$, 
the inequality
 $d(x, y)<\delta$
  implies 
$w(x, y)<\epsilon$. 
As a consequence of Theorem 
\ref{thm:main00}, we
prove that 
$\yoiset{w}$ 
is comeager in 
$\yocpm{X}$
(compare with the proof of
 \cite[Proposition 3]{MR4586584}). 
 
\begin{thm}\label{thm:mainiset}
Let $X$ be a metrizable space, 
and $w\in \yocpm{X}$. 
Then 
the set 
$\yoiset{w}$ 
is 
comeager in 
$\yocpm{X}$. 
\end{thm}

In the author's preprint 
\cite[Lemma 5.1]{Ishiki2020int}, 
it was shown  that 
the space 
$\met{X}$ is completely metrizable, 
(especially, Baire) 
 under the assumption that  
$X$ is second-countable and locally compact Hausdorff. 
Moreover, 
in \cite[Proposition 3]{MR4586584}, 
Koshino proved that 
 the set of all bounded metric $d\in \met{X}$
is completely metrizable under only the condition that 
$X$ is metrizable and $\sigma$-compact.
This proof is still  effective  for  $\met{X}$. 
Thus, 
Koshino's result  is a generalization of 
the author's lemma 
\cite[Lemma 5.1]{Ishiki2020int}. 
As a further generalization of these works, 
using  
Theorem 
\ref{thm:mainiset},
we obtain 
the following theorem 
(Theorem \ref{thm:main1baire}), 
 which 
 states that 
the space
$\met{X}$ 
of metrics is Baire. 

\begin{thm}\label{thm:main1baire}
Let
 $X$ 
 be a 
 metrizable space. 
Then 
$\met{X}$ is
comeager in 
$(\yocpm{X}, \metdis_{X})$. 
In particular, 
the space 
$(\met{X}, \metdis_{X})$ 
is 
Baire. 
\end{thm}

For a metrizable space 
$X$, 
we denote by 
$\compm{X}$
the set of all complete
 metrics
 in 
$\met{X}$. 
In the next result, 
by 
Theorem 
\ref{thm:mainiset}, 
we show that if $X$ is 
completely metrizable, 
then 
 $\compm{X}$
  is comeager in 
  $\met{X}$. 
In other words, 
generic  metrics on $X$ are
  complete. 
\begin{thm}\label{thm:maincomp}
Let 
$X$
 be a completely metrizable
 space. 
 Then 
 $\compm{X}$
 is comeager in 
 $(\yocpm{X}, \metdis_{X})$. 
Moreover, 
 the set
  $\compm{X}$
 is also  comeager in 
 $(\met{X}, \metdis_{X})$. 
 \end{thm}
 \begin{rmk}\label{rmk:nth}
 As shown in 
 \cite{NT1928} (see also \cite{Hausdorff1938}), 
 for every metrizable space $X$,
 we have $\met{X}=\compm{X}$ if and only if 
 $X$ is compact. 
 \end{rmk}

We also obtain 
non-Archimedean analogues of 
aforementioned theorems. 
Let us review ultrametrics (non-Archimedean metrics). 

A pseudometric $d\colon X\times X\to [0, \infty)$
is said to be 
a \yoemph{pseudo-ultrametric} or 
\yoemph{non-Archimedean pseudometric} if 
$d$ satisfies the so-called the strong triangle 
inequality 
$d(x, y)\le d(x, z)\lor d(z, y)$ for all $x, y, z\in X$, 
where the symbol 
``$\lor$'' means the maximum operator on $\rr$, i.e., 
$x\lor y=\max\{x, y\}$. 
A pair
 $(X, d)$ 
 is 
 called a \yoemph{pseudo-ultrametric space}. 
A pseudo-ultrametric $d$ on $X$ is 
called an \yoemph{ultrametric} or 
\yoemph{non-Archimedean metric} if 
the equality $d(x, y)=0$ implies $x=y$. 
Of course, every ultrametric is a metric.

A set $R$ is said to be 
a \yoemph{range set} if 
$R\yosub [0, \infty)$ 
and 
$0\in R$. 
We say that 
a range set $R$ is
\yoemph{characteristic} if 
for every $z\in (0, \infty)$, 
there exists 
$r\in R\setminus \{0\}$ such that 
$r< z$. 
This condition is equivalent to 
$\inf(R\setminus \{0\})=0$. 
A metric $d$ on $X$ is said to be 
\yoemph{$R$-valued}
 if 
$d(x, y)\in R$ for all 
$x, y\in X$.

For a topological space 
$X$, 
and for a range set 
$R$, 
we denote by 
$\yocmet{X}{R}$
the all
$R$-valued 
 continuous maps 
$d\colon X\times X\to [0, \infty)$
for which 
$d$ is a pseudo-ultrametric on $X$. 
We also denote by 
$\ult{X}{R}$
the all $R$-valued 
ultrametrics $d$ on $X$. 
Notice that 
$\ult{X}{R}\yosub \yocmet{X}{R}$. 
When considering 
non-Archimedean analogues, 
it is often more effective to limit the 
range of metrics (see, for example,  \cite{MR2754373}).
Namely, we investigate 
not only ($[0, \infty)$-valued) ultrametrics but also 
$R$-valued ultrametrics for an arbitrary  range set $R$.

For a range set $R$, 
a topological space $X$ is said to be 
\yoemph{
$R$-valued ultrametrizable} if 
$\ult{X}{R}\neq \emptyset$. 
When $R=[0, \infty)$, 
the space $X$ is simply said to be 
\yoemph{ultrametrizable}. 

\begin{rmk}
In  
  \cite[Proposition 2.14]{Ishiki2021ultra}, 
  it was shown that 
$X$ is ultrametrizable if and only if 
for every characteristic range set 
$R$, the space 
$X$ is 
$R$-valued ultrametrizable
($\ult{X}{R}\neq \emptyset$). 
\end{rmk}

We define 
$\umetdis_{X}^{R}\colon \yocmet{X}{R}^{2}\to  [0, \infty]$
by declaring that 
$\umetdis_{X}^{R}(d, e)$ is 
the infimum of all $\epsilon\in R$ such that 
$d(x, y)\le e(x, y)\lor \epsilon$ and 
$e(x, y)\le d(x, y)\lor \epsilon$
for all $x, y\in X$. 
Then 
$\umetdis_{X}^{R}$
 is an ultrametric on 
$\yocmet{X}{R}$ taking values in $[0, \infty]$. 
Similarly to $\metdis_{X}$, 
we can define the topology induced by 
$\umetdis_{X}^{R}$
using open balls.
In this paper, 
we represent 
the restricted metric 
$\umetdis_{X}^{R}|_{\ult{X}{R}^{2}}$
as the original  symbol 
$\umetdis_{X}^{R}$. 
In what follows, 
we consider that 
$\yocmet{X}{R}$ and 
$\ult{X}{R}$
 are equipped with
 the topologies induced by 
 $\umetdis_{X}^{R}$.

\begin{rmk}
Let $R$ be a range set, and
$X$ be an $R$-valued ultrametrizable space. 
Then  
we have the inclusions 
$\ult{X}{R}\yosub\met{X}$
and 
$\yocmet{X}{R}\yosub \yocpm{X}$. 
For every pair  
$d, e\in \yocmet{X}{R}$, 
we also obtain 
$\metdis_{X}(d, e)\le \umetdis_{X}^{R}(d, e)$. 
The topology generated  by
 $\umetdis_{X}^{R}(d, e)$ 
 is always  strictly stronger than that generated by 
$\metdis_{X}$. 
\end{rmk}

For a topological space 
$X$, 
for a range set $R$, 
and for an 
open covering 
$\yocover{C}$ 
of 
$X$, 
we define  
$\youles{\yocover{C}}{R}=\yocmet{X}{R}\cap \yoles{\yocover{C}}$. 

The next theorem is a non-Archimedean 
analogue of 
Theorem \ref{thm:main00}. 
The definition of ultraparacompactness 
can be found in Section \ref{sec:pre}. 

\begin{thm}\label{thm:mainna00}
Let 
$R$ be a range set, 
$X$  an
ultraparacompact Hausdorff space, 
and 
$\yocover{C}$
an open 
covering of 
$X$. 
Then 
the set 
$\youles{\yocover{C}}{R}$
 is 
open
and 
dense in 
$\yocmet{X}{R}$. 
\end{thm}

Let
$R$ be a range set, 
and 
 $X$  an $R$-valued  metrizable space, 
and take $w\in \yocpm{X}$. 
Notice  
that 
$w$ 
is not necessarily non-Archimedean. 
We define 
$\youiset{w}{R}$
the set of all 
$d\in \yocmet{X}{R}$
such that 
$1_{X}\colon (X, d)\to (X, w)$ is 
uniformly continuous. 

We also obtain  an analogue of 
Theorem \ref{thm:mainiset} for ultrametrics. 
\begin{thm}\label{thm:mainuiset}
Let
$R$ be a range set, 
and 
 $X$  an $R$-valued  ultrametrizable space, 
and take $w\in \yocpm{X}$
($w$ is not necessarily non-Archimedean). 
Then 
the set 
$\youiset{w}{R}$
is comeager in 
$\yocmet{X}{R}$. 
\end{thm}

The following theorem is corresponding to 
Theorem 
\ref{thm:main1baire}. 
\begin{thm}\label{thm:mainnabaire}
Let 
$R$ be a range set, 
$X$
 an 
 $R$-valued 
   ultrametrizable space. 
  Then 
  $\ult{X}{R}$
  is comeager in 
  $(\yocmet{X}{R}, \umetdis_{X}^{R})$. 
  In particular, 
the space 
$\left(\ult{X}{R}, \umetdis_{X}^{R}\right)$
 is 
Baire. 
\end{thm}

For a topological space 
$X$, 
and 
for a range set $R$, 
put 
$\ucompm{X}{R}=\ult{X}{R}\cap \compm{X}$.

The next result is an analogue of 
Theorem
\ref{thm:maincomp}. 
\begin{thm}\label{thm:mainnacomp}
Let 
$R$ be a range set, 
$X$ 
a completely 
 metrizable
 and $R$-valued ultrametrizable
 space. 
 Then 
 $\ucompm{X}{R}$
  is comeager in 
  $(\yocmet{X}{R}, \umetdis_{X}^{R})$. 
  Moreover, 
  the set
   $\ucompm{X}{R}$ 
   is also comeager in 
  $(\ult{X}{R}, \umetdis_{X}^{R})$. 
\end{thm}

\begin{rmk}
In  
  \cite[Propositions 2.14 and  2.17]{Ishiki2021ultra}, 
  it was proven  that 
  $X$ is completely metrizable and 
  ultrametrizable if and only if 
  for every characteristic range set $R$, 
  we have 
  $\ucompm{X}{R}\neq \emptyset$, i.e., 
  $X$ is 
  \yoemph{$R$-valued completely ultrametrizable}. 
\end{rmk}

\begin{rmk}
Let $X$ be a topological space, and 
$R$ be a non-characteristic range set. 
Then 
$X$ is 
$R$-valued ultrametrizable 
if and only if 
$X$ is discrete. 
In this case, we also have 
$\ult{X}{R}=\ucompm{X}{R}$. 
\end{rmk}

\begin{rmk}
Similarly to Remark \ref{rmk:nth}, 
in \cite[Corollary 1.3]{Ishiki2021ultra}, 
the author proved   that 
for every characteristic range set 
$R$, and 
for every  ultrametrizable space $X$, 
we have 
$\ult{X}{R}=\ucompm{X}{R}$
 if and only if 
$X$ is compact. 
\end{rmk}

All of our  main results will be proven in 
Section \ref{sec:proofs}
using  several  preliminaries in 
Section \ref{sec:pre}. 
In the final part of Section \ref{sec:proofs}, 
we will give some additional remarks. 

\section{Preliminaries}\label{sec:pre}

For the definitions of 
paracompactness, 
we refer to 
\cite[Section 20]{MR2048350}.

\begin{thm}\label{thm:les}
Let 
$X$
 be a paracompact Hausdorff space, 
 and 
 $\yocover{C}$
  be an open covering of $X$. 
  Then 
  $\yoles{\yocover{C}}\neq \emptyset$. 
  Equivalently, 
  there exists  a 
  continuous 
  pseudometric
  $D\in \yocpm{X}$
for which 
there exists 
 a 
 Lebesgue number 
 $r\in (0, \infty)$ 
 of 
 $\yocover{C}$ with respect to 
 $D$. 
\end{thm}
\begin{proof}
Theorem \ref{thm:les}
is  already known
(see \cite[Theorem 7.4]{MR0002515},  and 
see also 
 \cite[Remark 4]{MR3099433}, 
  \cite[Theorem 14]{MR0170323}, 
  and 
  \cite[Metrization Lemma 12, p.185]{MR0370454}). 
For the sake of self-containedness, 
we provide a proof.

For  a 
map
$h\colon X\to [0, \infty)$, 
we define 
$\supp (h)=\{\, x\in X\mid h(x)>0\, \}$
and call it the 
\yoemph{support}
 of 
 $h$. 
Put 
$\yocover{C}=\{C_{i}\}_{i\in I}$, 
and 
let 
$\{g_{a}\}_{a\in A}$
 be a locally finite 
  partition of unity 
 subordinated to  
 $\yocover{C}$
 (see
  \cite[Proposition 2]{MR0056905}
  and 
  \cite[Corollary 2.7.3]{MR3099433}). 
 Define  $\phi\colon X\to (0, 1]$ by 
 $\phi(x)=\sup_{a\in A}g_{a}(x)$. 
Since 
$\{g_{a}\}_{a\in A}$
is locally finite, 
the map 
$\phi$ 
is continuous. 
For every 
$a\in A$, 
we also 
define 
$f_{a}\colon X\to [0, 1]$ by 
 \[
f_{a}(x)=\frac{2}{\phi(x)}\cdot \min
\left\{g_{a}(x), \frac{\phi(x)}{2}
\right\}. 
 \]
 Then 
 the family 
 $\{f_{a}\}_{a\in A}$
satisfies that:
\begin{enumerate}[label=\textup{(P\arabic*)}]

\item\label{item:p1}
The family 
$\{\supp (f_{a})\}_{a\in A}$
 is
 a
  locally finite covering of 
 $X$, 
 and it is a refinement of
$\yocover{C}$; 

\item\label{item:p2}
The family 
$\{f_{a}^{-1}(1)\}_{a\in A}$
 is 
 a (closed) covering of
  $X$. 
 \end{enumerate}
 Indeed, 
 since for every $a\in A$, 
 we have 
 $\supp(f_{a})=\supp(g_{a})$, 
the condition 
 \ref{item:p1} is true. 
 For every 
 $x\in X$, 
 since 
 $\{g_{a}\}_{a\in A}$
 is locally finite, 
 there exists 
 $a_{0}\in A$
 such that 
 $\phi(x)=g_{a_{0}}(x)$, 
 and hence we have 
 $f_{a_{0}}(x)=1$. 
 Thus the condition \ref{item:p2} is 
 fulfilled. 

We now  denote by 
$\yoellsp{A}$
 the space of all 
 $w\colon A\to \rr$
 such that 
 $\sum_{a\in A}|w(a)|<\infty$, 
 and  denote by 
 $\yoelldis{*}$
 the 
 $\ell^{1}$-norm on 
 $\yoellsp{A}$. 
 For every 
 $a\in A$, we also denote 
 by
  $\yobasis_{a}$ 
 the element of 
  $\yoellsp{A}$ 
  such that 
  $\yobasis_{a}(p)=1$ if $p=a$; 
  otherwise 
  $\yobasis_{a}(p)=0$. 
We  
define a map 
 $\psi\colon X\to \yoellsp{A}$
  by 
  $\psi(x)=\sum_{a\in A}f_{a}(x)\cdot \yobasis_{a}$.
Since 
$\{\supp (f_{a})\}_{a\in A}$
 is locally finite
(the condition \ref{item:p1}), 
 the map 
 $\psi$ 
 is continuous. 
 We also define 
 a continuos pseudometric 
 $D\colon X\times X\to [0, \infty)$
  by 
  $D(x, y)=\yoelldis{\psi(x)-\psi(y)}$. 
In this setting, we see that 
  $D\in \yocpm{X}$. 
  
 Next we prove 
 that 
 $1$ is a 
 Lebesgue number 
 of 
 $\yocover{C}$
  with respect to 
 $D$. 
 Take $x\in X$. 
Then using 
\ref{item:p2}, 
we can find 
 $a_{0}\in A$ with 
 $f_{a_{0}}(x)=1$. 
 In this case, 
 every   
 $y\in U(x, 1; D)$
 satisfies that 
  $f_{a_{0}}(y)>0$. 
  Thus
  we obtain 
  $U(x, 1; D)\yosub \supp(f_{a_{0}})$,  
  and 
  the condition 
  \ref{item:p1}
  implies that  there exists 
  $i\in I$ 
  such that 
  $\supp(f_{a_{0}})\yosub C_{i}$. 
  Therefore we conclude that 
  $U(x, 1; D)\yosub C_{i}$. 
  This finishes the proof. 
\end{proof}

A topological space 
$X$ is 
said to be 
\yoemph{ultraparacompact} if 
every open covering $\yocover{C}$
of $X$ has 
a refinement covering of
 $X$ 
consisting of disjoint open subsets. 
Remark that 
a topological space is 
ultraparacompact if and only if 
it is paracompact and has 
covering dimension $0$
(see
\cite[Proposition 1.2]{MR0261565}). 
In particular, all ultrametrizable spaces are ultraparacompact. 

\begin{thm}\label{thm:ules}
Let 
$X$
 be an ultraparacompact Hausdorff space, 
 $R$  a range set, 
 and 
 $\yocover{C}$
  an open covering of $X$. 
  Then 
  $\youles{\yocover{C}}{R}\neq \emptyset$. 
  Equivalently, 
  there exists  a 
  continuous 
  pseudo-ultrametric
  $D\in \yocmet{X}{R}$
for which 
there exists 
 a 
 Lebesgue number 
 $r\in (0, \infty)$
 of 
 $\yocover{C}$
 with respect to $D$. 
\end{thm}
\begin{proof}
Since 
$X$ 
is 
ultraparacompact, 
there exists a 
disjoint 
open cover 
$\yocover{E}=\{O_{a}\}_{a\in A}$
subordinated to 
 $\yocover{C}$. 
We fix $r\in R$ and define 
 $D\in \yocmet{X}{R}$ by 
\[
D(x, y)=
\begin{cases}
0 & \text{if there exists $a\in A$ with $x, y\in O_{a}$;}\\
r &\text{otherwise.}
\end{cases}
\]

By the definition, 
the map $D$ is continuous on $X\times X$,  
and 
it is  a pseudo-ultrametric on $X$. 
Notice that 
for every $x\in X$, 
we have 
$U(x, r; D)=O_{a}$, 
where $O_{a}$ is an element 
of $\yocover{E}$ with 
$x\in O_{a}$. 
Thus $r$ is a 
Lebesgue number 
of 
$\yocover{C}$ with respect to 
$D$. 
This finishes the proof. 
\end{proof}

\begin{lem}\label{lem:comeager}
Every comeager subset of  a Baire space is 
itself Baire. 
\end{lem}
\begin{proof}
The lemma  follows from the definition of comeager sets. 
See also 
\cite[Theorem 1.15 and Proposition 1.23]{MR0431104}. 
\end{proof}

\begin{lem}\label{lem:daiji}
For every 
  topological space 
$X$, 
and for 
every range set 
$R$, 
the spaces
$(\yocpm{X}, \metdis_{X})$
and 
$(\yocmet{X}{R}, \umetdis_{X}^{R})$
  are complete metric spaces.
  In particular, these spaces are 
  Baire. 
\end{lem}
\begin{proof}
Let 
$\{d_{n}\}_{n\in \zz_{\ge 0}}$
 be a Cauchy sequence of $(\yocpm{X}, \metdis_{X})$. 
 Then it has a pointwise limit 
 $d\colon X\times X\to [0, \infty)$
 and it is also a pseudometric on 
 $X$. 
 Since 
 $\metdis_{X}$ is the supremum metric, 
 the map
  $d$
   is continuous. 
In the same way, 
   using 
   $ \metdis_{X}(d, e)\le \umetdis_{X}^{R}(d, e)$, 
   we  see that 
   $(\yocmet{X}{R}, \umetdis_{X}^{R})$
    is complete. 
 Similar arguments can be found in 
    the proofs of 
    \cite[Lemma 5.1]{Ishiki2020int}
    and \cite[Lemma 7.6]{Ishiki2021ultra}. 
    The latter part follows from the 
    Baire category theorem
    (see \cite[Corollary  25.4]{MR2048350}). 
\end{proof}

\begin{lem}\label{lem:gentop}
Let $X$ be a metrizable space. 
If $w\in \met{X}$, 
then 
we have 
$\yoiset{w}\yosub\met{X}$. 
\end{lem}
\begin{proof}
Take $d\in \yoiset{w}$. 
By the continuity of $d$ on $X\times X$, 
the identity $1_{X}\colon (X, d)\to (X, w)$ is an open map. 
Since $1_{X}\colon (X, d)\to (X, w)$ is continuous, 
we conclude that $1_{X}\colon (X, d)\to (X, w)$ is 
homeomorphism. 
Thus, the metric $d$ generates the same topology of 
$X$. This means that 
$\yoiset{w}\yosub \met{X}$. 
\end{proof}
Similarly, we also obtain the following 
non-Archimedean analogue. 
\begin{lem}\label{lem:ugentop}
Let $R$ be a range set,  and 
 $X$ be an $R$-valued  ultrametrizable space. 
If $w\in \met{X}$
($w$ is not necessarily non-Archimedean), 
then we have $\youiset{w}{R}\yosub\ult{X}{R}$. 
\end{lem}

\section{Proofs of Main results}\label{sec:proofs}

\subsection{Proofs of Archimedean main  results}

Now we provide 
 proofs of 
Theorems
 \ref{thm:main00}--\ref{thm:maincomp}. 

\begin{proof}[Proof of Theorem \ref{thm:main00}]
Let 
$X$ be a 
paracompact Hausdorff space, 
and 
$\yocover{C}$
an open 
covering of 
$X$. 

First let us prove that 
$\yoles{\yocover{C}}$ is open. 
Take an arbitrary member 
 $d\in \yoles{\yocover{C}}$, 
and let $r$ be a Lebesgue number of 
$\yocover{C}$ with respect to 
$d$. 
Fix  
$\epsilon\in (0, \infty)$
with 
$\epsilon <r$. 
For every 
 $e\in \yocpm{X}$ such that 
$\metdis_{X}(d, e)<\epsilon$, 
we 
put
 $r^{\prime}=r-\epsilon>0$. 
Then  we have 
$U(x, r^{\prime}; e)\yosub U(x, r; d)$
by $d(x, y)\le e(x, y)+\epsilon$ 
for all
 $x, y\in X$. 
Thus, 
$r^{\prime}$ is a
Lebesgue number of 
$\yocover{C}$
with respect to $e$, 
and hence 
$e\in \yoles{\yocover{C}}$. 
Therefore
$\yoles{\yocover{C}}$ is open. 

Next we show that 
$\yoles{\yocover{C}}$ is 
dense. 
Using Theorem 
\ref{thm:les}, we see that 
$\yoles{\yocover{C}}\neq \emptyset$. 
Fix 
$e\in \yoles{\yocover{C}}$, and 
let $l$ be a Lebesgue number of 
$\yocover{C}$ with respect to $e$. 
Define a pseudometric  $h$ on $X$ by 
$h(x, y)=\min\{e(x, y), 1\}$. 
In this setting, the number 
$r=\min\{l, 2^{-1}\}$ is  a 
Lebesgue number of 
$\yocover{C}$
with respect $h$. 
Namely, $h\in \yoles{\yocover{C}}$. 
Take an arbitrary member 
$d\in \yocpm{X}$, 
and an arbitrary number 
$\epsilon\in (0, \infty)$. 
Put 
$p=d+\epsilon\cdot h\in \yocpm{X}$. 
Since 
$h(a, b)\le 1$ for all $a, b\in X$, 
it is true  that 
$p$
satisfies 
$\metdis_{X}(d, p)\le \epsilon$. 
Let us show that 
$p\in \yoles{\yocover{C}}$. 
Take 
an arbitrary point 
$x\in X$. 
From the fact that $\epsilon\cdot h(a, b)\le p(a, b)$
for all $a, b\in X$,  
it follows that 
$U(x, \epsilon r; p)
\yosub
U(x, \epsilon r ; \epsilon h)$. 
By the definition of open balls,  
we have 
$
U(x, \epsilon r ; \epsilon h)
=U(x, r; h)$. 
Thus
we obtain  $U(x, \epsilon r; p)\yosub U(x, r; h)$, 
and hence 
$\epsilon \cdot r$ is a Lebesgue number 
of $\yocover{C}$ with respect to 
$p(=d+\epsilon \cdot h)$. 
Therefore
$p\in \yoles{\yocover{C}}$, 
and we then conclude that 
$\yoles{\yocover{C}}$ is dense
in $\yocpm{X}$. 
This completes the 
proof of  Theorem \ref{thm:main00}. 
\end{proof}

\begin{proof}[Proof of Theorem \ref{thm:mainiset}]
Let $X$ be a metrizable space, 
and $w\in \yocpm{X}$. 

For each $n\in \zz_{\ge 0}$, 
put 
$\yocover{O}_{n}=\{\, U(x, 2^{-n}; w)\mid x\in X\, \}$, 
and 
$S=\bigcap_{n\in \zz_{\ge 0}}\yoles{\yocover{O}_{n}}$.
Since 
every metrizable 
 space is 
paracompact
(see 
\cite{MR0026802}
and
\cite{MR0236876}), 
we can apply  Theorem \ref{thm:main00}
to $X$ and each $\yocover{O}_{n}$. 
Then we observe that 
$\yoles{\yocover{O}_{n}}$
 is open and dense. 
Thus,  to prove the  theorem, 
it suffices to  show   that 
$S\yosub \yoiset{w}$. 
Take an arbitrary member 
$d\in S$
and an arbitrary number 
$\epsilon \in (0, \infty)$. 
Now let us  verify that 
there exists $\delta\in (0, \infty)$ such that 
the inequality $d(x, y)<\delta$
implies 
$w(x, y)<\epsilon$ for all 
$x, y\in X$.
Take a sufficient large number 
$m\in \zz_{\ge 0}$ such that 
$2^{-m}\le \epsilon$. 
By 
$d\in \yoles{\yocover{O}_{m+1}}$, 
we can find a Lebesgue number 
$\delta$ of 
$\yocover{O}_{n+1}$ with respect 
to $d$. 
Thus, 
there exists 
$z\in X$
 such that 
$U(x, \delta; d)\yosub U(z, 2^{-m-1}; w)$. 
Due to  $x\in  U(z, 2^{-m-1}; w)$, 
the triangle inequality for 
$w$ 
implies
$U(z, 2^{-m-1}; w)\yosub U(x, 2^{-m}; w)$. 
Hence 
we obtain 
$U(x, \delta; d)\yosub U(x, 2^{-m}; w)$. 
Then the inequality 
$d(x, y)<\delta$ 
implies that 
$w(x, y)<2^{-m}\le \epsilon$. 
Therefore
$S\yosub \yoiset{w}$. 
\end{proof}

\begin{proof}[Proof of Theorem \ref{thm:main1baire}]
Let $X$ be a metrizable space. 
Lemma \ref{lem:gentop} states that 
$\yoiset{w}\yosub \met{X}$. 
Due to Theorem \ref{thm:mainiset}, 
we conclude that $\met{X}$ is comeager in 
$\yocpm{X}$. 
Therefore 
Lemmas  \ref{lem:comeager}
and \ref{lem:daiji}
prove that  
$\met{X}$ is Baire. 
This completes  the proof of Theorem \ref{thm:main1baire}. 
\end{proof}

\begin{proof}[Proof of Theorem \ref{thm:maincomp}]
Let 
$X$ be a completely metrizable space. 
Then,  
we can take a complete metric 
$w\in \compm{X}$. 
Let us show that 
$\yoiset{w}\yosub \compm{X}$. 
Due to Lemma \ref{lem:gentop}, 
we have $\yoiset{w}\yosub \met{X}$. 
Thus it suffices to  verify that 
every $d\in \yoiset{w}$ is complete. 
Take 
a Cauchy sequence 
$\{x_{i}\}_{i\in \zz_{\ge 0}}$ of $(X, d)$. 
Since 
the map 
$1_{X}\colon (X, d)\to (X, w)$ is uniformly continuous, 
 the sequence 
 $\{x_{i}\}_{i\in \zz_{\ge 0}}$ 
 is also Cauchy in 
 $(X, w)$. 
 By the completeness $(X, d)$, 
 the sequence 
 $\{x_{i}\}_{i\in \zz_{\ge 0}}$ 
 has a limit point, say $p$. 
 Since $d$ generates the same topology  of $X$, 
we see that  $p$ is also a limit point of 
 $\{x_{i}\}_{i\in \zz_{\ge 0}}$
 in 
 $(X, d)$. 
 Namely, 
 the space 
 $(X, d)$ is complete. 
Therefore
 $\yoiset{w}\yosub \compm{X}$. 
 This means that 
 $\compm{X}$ 
 is 
 comeager in 
 $\yocpm{X}$. 
  This finishes the proof of Theorem \ref{thm:maincomp}. 
\end{proof}

\subsection{Proofs of Non-Archimedean main results}

\begin{proof}[Proof of Theorem \ref{thm:mainna00}]
Let 
$X$ be a 
paracompact Hausdorff space, 
and 
$\yocover{C}$
an open 
covering of 
$X$. 
The proof is parallel to 
that of Theorem \ref{thm:main00}. 

First let us prove that 
$\youles{\yocover{C}}{R}$ is open. 
Take an arbitrary member 
 $d\in \youles{\yocover{C}}{R}$, 
 and let 
 $r$
  be a Lebesgue number 
  of $\yocover{C}$ with respect to 
  $d$. 
 Fix  
$\epsilon\in (0, \infty)$
with 
$\epsilon <r$. 
For every 
 $e\in \yocmet{X}{R}$ such that 
$\umetdis_{X}(d, e)< \epsilon$, 
 we have 
 $d(x, y)\le e(x, y)\lor \epsilon$, 
 and hence 
$U(x, r; e)\yosub U(x, r; d)$. 
Thus, 
$r$ is also a Lebesgue number 
of $\yocover{C}$
with respect to $e$, 
and hence 
$e\in \youles{\yocover{C}}{R}$. 
Therefore  
$\youles{\yocover{C}}{R}$ 
is open in 
$\yocmet{X}{R}$. 

Next we prove that 
$\youles{\yocover{C}}{R}$  is 
dense in $\yocmet{X}{R}$. 
Using Theorem 
\ref{thm:ules}, we see that 
$\youles{\yocover{C}}{R}\neq \emptyset$. 
Take an arbitrary member 
$d\in \yocmet{X}{R}$ and 
an arbitrary number 
$\epsilon\in (0, \infty)$. 
Fix 
$e\in \youles{\yocover{C}}{R}$, and 
let $r$ be a Lebesgue number of 
$\yocover{C}$  with respect to 
$e$. 
We divide the proof into two parts. 

\yocase{1}{$R$ is characteristic}
In this case, we can take 
$\eta\in R\setminus \{0\}$ such that 
$\eta<\min\{\epsilon, r\}$. 
Put 
$h(x, y)=\min\{e(x, y), \eta\}$. 
Then 
$h\in \yocmet{X}{R}$ and 
$h\in \youles{\yocover{C}}{R}$. 
We
put 
$p=d\lor h\in \yocmet{X}{R}$. 
Let us show $\umetdis_{X}^{R}(d, p)\le \epsilon$. 
For every pair $a, b\in X$, 
we have 
$p(a, b)=d(a, b)\lor h(a, b)\le d(a, b)\lor \eta\le 
d(a, b)\lor \epsilon$. 
We also have 
$d(a, b)\le h(a, b)\le h(a, b)\lor\epsilon$. 
Then 
$\umetdis_{X}^{R}(d, p)\le \epsilon$. 
Since 
$h(a, b)\le p(a, b)$
for all $a, b\in X$, 
we have 
$U(x, \eta; p)
\yosub
U(x, \eta; h)$. 
By the definition of 
$h$, we have 
$h(a, b)<\eta$ 
if and only if 
$e(a, b)<\eta$
for all $a, b\in X$. 
Thus 
$U(x, \eta; h)
=
U(x, \eta; e)$. 
Due to 
$\eta<r$, we have 
$U(x, \eta; e)
\yosub 
U(x, r; e)$. 
Finally, we obtain 
$U(x, \eta; p)
\yosub U(x, r; e)$. 
Hence 
$\eta$ is a Lebesgue number 
of $\yocover{C}$ with respect to 
$p(=d\lor h)$. 
Thus 
$p\in \youles{\yocover{C}}{R}$.

\yocase{2}{$R$ is not characteristic}
Under this assumption, 
we have the inequality
$0<\inf(R\setminus \{0\})$. 
Put 
$\delta=(1/2)\cdot \inf(R\setminus \{0\})$. 
Then  we see that $\delta>0$ and 
$U(x; \delta; d)=\{x\}$. 
Hence 
$\delta$ is 
a Lebesgue number of 
$\yocover{C}$ with respect to 
$d$. 
This means that 
$d\in \youles{\yocover{C}}{R}$. 
Namely, in this case, 
we have 
$\youles{\yocover{C}}{R}=\yocmet{X}{R}$.

Therefore, in any case,   we 
 conclude that the set 
$\youles{\yocover{C}}{R}$ is dense
in $\yocmet{X}{R}$. 
This completes the 
proof of Theorem \ref{thm:mainna00}. 
\end{proof}

\begin{proof}[Proof of Theorem \ref{thm:mainuiset}]
Let $R$ be a range set, 
and 
 $X$ be an $R$-valued  ultrametrizable space. 
Take $w\in \yocpm{X}$. 
For each $n\in \zz_{\ge 0}$, 
put 
$\yocover{O}_{n}=\{\, U(x, 2^{-n}; w)\mid x\in X\, \}$, 
and 
$S=\bigcap_{n\in \zz_{\ge 0}}\youles{\yocover{O}_{n}}{R}$. 
Then, 
similarly to 
 the proof  of Theorem \ref{thm:mainiset}, 
 we
obtain  
$S\yosub \youiset{w}{R}$. 
Since 
every 
ultrametrizable 
 space is 
ultraparacompact
(see \cite[Proposition 1.2 and Corollary 1.4]{MR0261565} and 
\cite[Theorem II]{MR80905}), 
we can apply 
Theorem \ref{thm:mainna00}
to $X$ and each $\yocover{O}_{n}$. 
Then we see that 
each 
$\youles{\yocover{O}_{n}}{R}$ is 
open and dense in 
$\yocmet{X}{R}$. 
Hence $\youiset{w}{R}$ is comeager in 
$\yocmet{X}{R}$. 
This finishes the proof. 
\end{proof}

\begin{proof}[Proof of Theorem \ref{thm:mainnabaire}]
The proof is similar to 
that of Theorem \ref{thm:main1baire}. 
Let 
$R$ be a range set,  and 
$X$
be an $R$-valued   ultrametrizable space. 
Since $X$ is metrizable, 
we obtain $\met{X}\neq \emptyset$. 
Fix $w\in \met{X}$. 
Lemma \ref{lem:ugentop} shows 
that 
$\youiset{w}{R}\yosub \ult{X}{R}$. 
Hence,  
using Theorem \ref{thm:mainuiset},
the set  $\ult{X}{R}$ is comeager in 
$\yocmet{X}{R}$. 
Therefore 
Lemmas  \ref{lem:comeager}
and \ref{lem:daiji}
prove that  
$\ult{X}{R}$ is Baire. 
\end{proof}

\begin{proof}[Proof of Theorem \ref{thm:mainnacomp}]
Let 
$R$ be a range set, and 
$X$ be 
a completely 
 metrizable
 and 
 $R$-valued 
 ultrametrizable
 space. 
 Since $X$ is completely metrizable, 
 we obtain  $\compm{X}\neq \emptyset$. 
 Fix $w\in \compm{X}$. 
In the same way as
the proof of 
Theorem \ref{thm:maincomp}, 
we can prove 
that 
$\youiset{w}{R}\yosub \ucompm{X}{R}$. 
Using
Theorem \ref{thm:mainuiset}, 
we conclude that 
$\ucompm{X}{R}$ is 
comeager in 
$\yocmet{X}{R}$. 
This completes the proof of  Theorem \ref{thm:mainnacomp}. 
\end{proof}


\subsection{Additional remarks}
As shown in \cite[Proposition 3]{MR4586584}
(see also \cite[Lemma 5.1]{Ishiki2020int}), 
if $X$ is metrizable and $\sigma$-compact, 
then $\met{X}$ is completely metrizable. 
In the first version of the preprint of this paper, 
the author conjectured that 
the  inverse of this result is true. 
Recently,  
Koshino 
solved this conjecture.
Namely, he proved  that, 
for every separable metrizable 
$X$, 
the space
 $\met{X}$ 
 is completely metrizable 
 if and only if 
$X$ is 
$\sigma$-compact
(see \cite[Theorem and Remark]{koshino2024borel}). 
Now we make a question on spaces of metrics. 

\begin{ques}
For a metrizable  sapce $X$, 
is 
$\met{X}$ Borel in 
$\yocpm{X}$?
If this is  the case, 
what is the Borel hierarchy
of 
$\met{X}$ in 
$\yocpm{X}$?
\end{ques}


\begin{ac}
The author wishes to express his 
deepest gratitude 
to 
all  members of 
Photonics Control Technology Team (PCTT) in
 RIKEN, 
where
the majority  of the paper were written, 
for their invaluable   supports. 
Special thanks are extended to 
  the 
Principal Investigator of PCTT,
Satoshi Wada for 
the encouragement and support that 
transcended disciplinary boundaries. 

The author would like to thank 
Katsuhisa Koshino for helpful 
comments. 

This work is partially supported by JSPS 
KAKENHI Grant Number 
JP24KJ0182. 
\end{ac}




\bibliographystyle{myplaindoidoi}
\bibliography{bibtex/compmet.bib}

\providecommand{\bysame}{\leavevmode\hbox to3em{\hrulefill}\thinspace}
\providecommand{\MR}{\relax\ifhmode\unskip\space\fi MR }
\providecommand{\MRhref}[2]{%
  \href{http://www.ams.org/mathscinet-getitem?mr=#1}{#2}
}
\providecommand{\href}[2]{#2}
\begin{thebibliography}{10}

\bibitem{baire1899fonctions}
R.~Baire, \emph{Sur les fonctions de variables r{\'e}elles}, Annali di
  Matematica Pura ed Applicata (1898-1922) \textbf{3} (1899), 1--123,
  DOI:{10.1007/BF02419243}.

\bibitem{Banach1931}
S.~Banach, \emph{\"{U}ber die baire'sche kategorie gewisser funktionenmengen},
  Stud. Math. \textbf{3} (1931), no.~1, 174--179, DOI:{10.4064/sm-3-1-174-179}.

\bibitem{vceretkova1997certain}
S.~{\v{C}}eretkov{\'a}, J.~Fulier, and J.~T. T{\'o}th, \emph{On the certain
  subsets of the space of metrics}, Acta Academiae Paedagogicae Agriensis,
  Sectio Mathematicae \textbf{24} (1997), 111--115.

\bibitem{MR80905}
J.~de~Groot, \emph{Non-{A}rchimedean metrics in topology}, Proc. Amer. Math.
  Soc. \textbf{7} (1956), 948--953, DOI:{10.2307/2033568}. \MR{80905}

\bibitem{MR0261565}
R.~L. Ellis, \emph{Extending continuous functions on zero-dimensional spaces},
  Math. Ann. \textbf{186} (1970), 114--122, DOI:{10.1007/BF01350686}.
  \MR{261565}

\bibitem{MR2754373}
S.~Gao and C.~Shao, \emph{Polish ultrametric {U}rysohn spaces and their
  isometry groups}, Topology Appl. \textbf{158} (2011), no.~3, 492--508,
  DOI:{10.1016/j.topol.2010.12.003}. \MR{2754373}

\bibitem{MR3518958}
Gilles Godefroy, \emph{A survey on {L}ipschitz-free {B}anach spaces}, Comment.
  Math. \textbf{55} (2015), no.~2, 89--118, DOI:{10.14708/cm.v55i2.1104}.
  \MR{3518958}

\bibitem{Hausdorff1938}
F.~Hausdorff, \emph{{E}rweiterung einer stetigen {A}bbildung}, Fund. Math.
  \textbf{43} (1938), 40--47, DOI:{10.4064/fm-30-1-40-47}.

\bibitem{MR0431104}
R.~C. Haworth and R.~A. McCoy, \emph{Baire spaces}, Dissertationes Math.
  (Rozprawy Mat.) \textbf{141} (1977), 73, \url{http://eudml.org/doc/268479}.
  \MR{431104}

\bibitem{MR0170323}
J.~R. Isbell, \emph{Uniform spaces}, Mathematical Surveys, vol. No. 12,
  American Mathematical Society, Providence, RI, 1964. \MR{170323}

\bibitem{Ishiki2020int}
Y.~Ishiki, \emph{An interpolation of metrics and spaces of metrics},  (2020),
  preprint, arXiv:2003.13277.

\bibitem{Ishiki2021ultra}
\bysame, \emph{An embedding, an extension, and an interpolation of
  ultrametrics}, $p$-Adic Numbers Ultrametric Anal. Appl. \textbf{13} (2021),
  no.~2, 117--147, DOI:{10.1134/S2070046621020023}. \MR{4265905}

\bibitem{Ishiki2021dense}
\bysame, \emph{On dense subsets in spaces of metrics}, Colloq. Math.
  \textbf{170} (2022), no.~1, 27--39, DOI:{10.4064/cm8580-9-2021}. \MR{4460212}

\bibitem{MR4527953}
\bysame, \emph{Extending proper metrics}, Topology Appl. \textbf{325} (2023),
  Paper No. 108387, 11 pages, DOI:{10.1016/j.topol.2022.108387}. \MR{4527953}

\bibitem{Ishiki2023disco}
\bysame, \emph{On comeager sets of metrics whose ranges are disconnected},
  Topology Appl. \textbf{327} (2023), Paper No. 108442, 10 pages,
  DOI:{10.1016/j.topol.2023.108442}. \MR{4548505}

\bibitem{Ishiki2023sr}
\bysame, \emph{Strongly rigid metrics in spaces of metrics}, Topology Proc.
  \textbf{63} (2024), 125--148, arXiv:2210.02170.

\bibitem{MR0370454}
John~L. Kelley, \emph{General topology}, Graduate Texts in Mathematics, vol.
  No. 27, Springer-Verlag, New York-Berlin, 1975, Reprint of the 1955 edition
  [Van Nostrand, Toronto, Ont.]. \MR{370454}

\bibitem{MR4586584}
K.~Koshino, \emph{Recognizing the topologies of spaces of metrics with the
  topology of uniform convergence}, Bull. Pol. Acad. Sci. Math. \textbf{70}
  (2022), no.~2, 165--171, DOI:{10.4064/ba220523-18-4}. \MR{4586584}

\bibitem{koshino2022topological}
\bysame, \emph{The topological type of spaces consisting of certain metrics on
  locally compact metrizable spaces with the compact-open topology}, arXiv
  preprint arXiv:2202.08615 (2022).

\bibitem{koshino2024borel}
\bysame, \emph{On the {B}orel complexity and the complete metrizability of
  spaces of metrics},  (2024), preprint, arXiv:2403.07421.

\bibitem{Mazurkiewicz1931}
S.~Mazurkiewicz, \emph{Sur les fonctions non d\'{e}rivables}, Stud. Math.
  \textbf{3} (1931), no.~1, 92--94, DOI:{10.4064/sm-3-1-92-94}.

\bibitem{MR0056905}
E.~Michael, \emph{A note on paracompact spaces}, Proc. Amer. Math. Soc.
  \textbf{4} (1953), 831--838, DOI:{10.2307/2032419}. \MR{56905}

\bibitem{NT1928}
V.~Niemytzki and A.~Tychonoff, \emph{{B}eweis des {S}atzes, dass ein
  metrisierbarer {R}aum dann und nur dann kompakt ist, wenn er in jeder
  {M}etrik vollst\"{a}ndig ist}, Fund. Math. \textbf{12} (1928), 118--120,
  DOI:{10.4064/fm-12-1-118-120}.

\bibitem{MR0584443}
J.~C. Oxtoby, \emph{Measure and category}, second ed., Graduate Texts in
  Mathematics, vol.~2.

\bibitem{MR0236876}
M.~E. Rudin, \emph{A new proof that metric spaces are paracompact}, Proc. Amer.
  Math. Soc. \textbf{20} (1969), 603, DOI:{10.2307/2035708}. \MR{236876}

\bibitem{MR3099433}
K.~Sakai, \emph{Geometric aspects of general topology}, Springer Monographs in
  Mathematics, Springer, Tokyo, 2013, DOI:{10.1007/978-4-431-54397-8}.
  \MR{3099433}

\bibitem{MR0010962}
M.~E. Shanks, \emph{The space of metrics on a compact metrizable space}, Amer.
  J. Math. \textbf{66} (1944), 461--469, DOI:{10.2307/2371909}. \MR{10962}

\bibitem{smith2023lipschitz}
R.~J. Smith and F.~Talimdjioski, \emph{Lipschitz-free spaces over properly
  metrisable spaces and approximation properties}, arXiv preprint
  arXiv:2308.14121 (2023).

\bibitem{MR0026802}
A.~H. Stone, \emph{Paracompactness and product spaces}, Bull. Amer. Math. Soc.
  \textbf{54} (1948), 977--982, DOI:{10.1090/S0002-9904-1948-09118-2}.
  \MR{26802}

\bibitem{MR4641123}
Filip Talimdjioski, \emph{Lipschitz-free spaces over {C}antor sets and
  approximation properties}, Mediterr. J. Math. \textbf{20} (2023), no.~6,
  Paper No. 302, 16, DOI:{10.1007/s00009-023-02503-1}. \MR{4641123}

\bibitem{MR0002515}
J.~W. Tukey, \emph{Convergence and {U}niformity in {T}opology}, Annals of
  Mathematics Studies, vol. No. 2, Princeton University Press, Princeton, NJ,
  1940. \MR{2515}

\bibitem{MR1407287}
R.~W. Vallin, \emph{More on the metric space of metrics}, Real Anal. Exchange
  \textbf{21} (1995/96), no.~2, 739--742. \MR{1407287}

\bibitem{MR1691182}
A.~M. Vershik, \emph{The universal {U}ryson space, {G}romov's metric triples,
  and random metrics on the series of natural numbers}, Uspekhi Mat. Nauk
  \textbf{53} (1998), no.~5(323), 57--64,
  DOI:{10.1070/rm1998v053n05ABEH000069}. \MR{1691182}

\bibitem{MR2086637}
\bysame, \emph{Random metric spaces and universality}, Uspekhi Mat. Nauk
  \textbf{59} (2004), no.~2(356), 65--104,
  DOI:{10.1070/RM2004v059n02ABEH000718}. \MR{2086637}

\bibitem{MR1198097}
T.~\v{S}al\'{a}t, J.~T\'{o}th, and L.~Zsilinszky, \emph{On cardinality of sets
  of metrics generating metric spaces of prescribed properties}, Ann. Univ.
  Sci. Budapest. E\"{o}tv\"{o}s Sect. Math. \textbf{35} (1992), 15--21.
  \MR{1198097}

\bibitem{MR1205515}
\bysame, \emph{Metric space of metrics defined on a given set}, Real Anal.
  Exchange \textbf{18} (1992/93), no.~1, 225--231. \MR{1205515}

\bibitem{MR1268864}
\bysame, \emph{On the structure of the space of metrics defined on a given
  set}, Real Anal. Exchange \textbf{19} (1993/94), no.~1, 321--327.
  \MR{1268864}

\bibitem{MR2048350}
S.~Willard, \emph{General topology}, Dover Publications, Inc., Mineola, NY,
  2004, Reprint of the 1970 original [Addison-Wesley, Reading, MA; MR0264581].
  \MR{2048350}

\end{thebibliography}



\end{document}